\documentclass[10pt]{article}

\usepackage{amsmath,amsthm,amsfonts,latexsym,amsopn,verbatim,amscd,amssymb}

\theoremstyle{plain}
\newtheorem{theorem}{Theorem}[section]
\newtheorem{proposition}[theorem]{Proposition}
\newtheorem{corollary}[theorem]{Corollary}
\theoremstyle{definition}
\numberwithin{equation}{section}

\DeclareMathOperator{\spec}{Spec}

\DeclareMathOperator{\cent}{Cent}

\newcommand{\bnum}{\begin{enumerate}}
\newcommand{\enum}{\end{enumerate}}

\begin{document}

\title{Finite groups whose commuting graphs  are integral }
\author{Jutirekha Dutta and Rajat Kanti Nath\footnote{Corresponding author}   }
\date{}
\maketitle
\begin{center}\small{ Department of Mathematical Sciences,\\ Tezpur
University,  Napaam-784028, Sonitpur, Assam, India.\\
Emails: jutirekhadutta@yahoo.com, rajatkantinath@yahoo.com }
\end{center}


\smallskip

\noindent {\small{\textbf{Abstract:}  A finite non-abelian group $G$ is called commuting integral  if the commuting graph of $G$ is integral. In this paper, we show that  a finite group  is commuting integral if its central factor is isomorphic to  ${\mathbb{Z}}_p \times {\mathbb{Z}}_p$ or $D_{2m}$, where $p$ is any prime integer and $D_{2m}$ is the dihedral group of order $2m$. }

\bigskip



\section{ Introduction}

Let $G$ be a non-abelian group with center $Z(G)$. The commuting graph of $G$, denoted by $\Gamma_G$, is a simple undirected graph whose vertex set is $G\setminus Z(G)$, and two vertices $x$ and $y$ are adjacent if and only if $xy = yx$. In recent years, many mathematicians have considered  commuting graph of different finite groups and studied  various graph theoretic aspects (see \cite{amr06,bbhR09,iJ07,iJ08,mP13,par13}). A finite non-abelian group $G$ is called \textit{commuting integral } if the commuting graph of $G$ is integral. It is natural to ask     which finite groups are commuting integral. In this paper, we compute the spectrum of the commuting graphs of   finite groups whose central factors are isomorphic to ${\mathbb{Z}}_p \times {\mathbb{Z}}_p$, for any prime integer $p$, or $D_{2m}$, the dihedral group of order $2m$. Our computation reveals that  those groups are commuting integral.   


Recall that the spectrum of a graph  ${\mathcal{G}}$ denoted by $\spec({\mathcal{G}})$ is the set $\{\lambda_1^{k_1}, \lambda_2^{k_2},\dots$, $\lambda_n^{k_n}\}$, where $\lambda_1,  \lambda_2, \dots, \lambda_n$ are the eigenvalues of the adjacency matrix of $\mathcal{G}$ with multiplicities $k_1, k_2, \dots, k_n$ respectively. A graph ${\mathcal{G}}$ is called integral if $\spec({\mathcal{G}})$ contains only integers. It is well known that the complete graph $K_n$ on $n$ vertices is integral and $\spec(K_n) = \{(-1)^{n - 1}, (n - 1)^1\}$. Further, if $\mathcal{G} = K_{m_1}\sqcup K_{m_2}\sqcup\cdots  \sqcup K_{m_l}$, where $K_{m_i}$ are complete graphs on $m_i$ vertices for $1 \leq i \leq l$, then 
\[
\spec(\mathcal{G}) = \{(-1)^{\underset{i = 1}{\overset{l}{\sum}}m_i - l},\, (m_1 - 1)^1,\, (m_2 - 1)^1,\, \dots,\, (m_l - 1)^1\}.
\] 
  The notion of integral graph was introduced by 
Harary and  Schwenk \cite{hS74} in the year 1974. Since then many mathematicians have considered integral graphs, see for example \cite{aV09,iV07,wLH05}.
 A very impressive survey on integral graphs can be found in \cite{bCrS03}. Ahmadi et. al noted  that   integral graphs have some interest for designing the network topology of perfect state transfer networks, see  \cite{anb09} and the references there in.

 For any element $x$ of a group $G$, the set $C_G(x) = \{y \in G : xy = yx\}$ is called the centralizer of  $x $ in $G$. Let $|\cent(G)| = |\{C_G(x) : x \in G\}|$, that is the number of distinct centralizers in $G$. A group $G$ is called an $n$-centralizer group if $|\cent(G)| = n$. In \cite{bG94}, Belcastro and  Sherman  characterized finite $n$-centralizer groups for $n = 4, 5$. As a   consequence of our results,  we  show  that   $4, 5$-centralizer finite groups are commuting integral. Further, we show that   a finite $(p + 2)$-centralizer $p$-group is commuting integral for any prime $p$.

\section{Main results and consequences}

 
 
We begin this section with the following theorem.  
 
\begin{theorem}\label{main2}
Let $G$ be a finite group such that $\frac{G}{Z(G)} \cong {\mathbb{Z}}_p \times {\mathbb{Z}}_p$, where $p$ is a prime integer. Then 
\[
\spec(\Gamma_G) = \{(-1)^{(p^2 - 1)|Z(G)| - p - 1}, ((p - 1)|Z(G)| - 1)^{p + 1}\}.
\]
\end{theorem}
\begin{proof}
Let $|Z(G)| = n$ then since $\frac{G}{Z(G)}\cong {\mathbb{Z}}_p\times {\mathbb{Z}}_p$ we have  $\frac{G}{Z(G)} = \langle aZ(G), bZ(G) : a^p, b^p, aba^{-1}b^{-1} \in Z(G)\rangle$, where $a, b \in G$ with $ab \ne ba$. Then for any $z \in Z(G)$, we have
\begin{align*}
C_G(a) &= C_G(a^iz) \,\,\, = Z(G) \sqcup aZ(G) \sqcup \cdots \sqcup a^{p -1}Z(G) \text{ for } 1 \leq i \leq p - 1,\\
C_G(a^jb) &= C_G(a^jbz) = Z(G) \sqcup a^jbZ(G) \sqcup \cdots \sqcup a^{(p -1)j}b^{p - 1}Z(G) \text{ for } 1 \leq j \leq p.
\end{align*}
These are the only  centralizers of non-central elements of $G$. Also note that these centralizers are abelian subgroups of $G$. Therefore
\[
\Gamma_G = K_{|C_G(a)\setminus Z(G)|} \sqcup (\underset{j = 1}{\overset{p}{\sqcup}} K_{|C_G(a^jb)\setminus Z(G)|}).
\]
 Thus $\Gamma_G = K_{(p - 1)n} \sqcup (\underset{j = 1}{\overset{p}{\sqcup}} K_{(p - 1)n})$,  since  $|C_G(a)| = pn$ and $|C_G(a^jb)| = pn$ for $1 \leq j \leq p$ where as usual $K_m$ denotes the complete graph with $m$ vertices. That is, $\Gamma_G =  \underset{j = 1}{\overset{p + 1}{\sqcup}} K_{(p - 1)n}$.  Hence the result follows.
\end{proof}
Above theorem shows that $G$ is commuting integral if  the central factor of $G$ is isomorphic to ${\mathbb{Z}}_p \times {\mathbb{Z}}_p$ for any prime integer $p$. Some consequences of Theorem \ref{main2} are given below.
\begin{corollary}
Let $G$ be a non-abelian group of order $p^3$, for any prime $p$, then  
\[
\spec(\Gamma_G) = \{(-1)^{p^3 - 2p - 1}, (p^2 - p - 1)^{p + 1}\}.
\]
Hence, $G$ is commuting integral.
\end{corollary}

\begin{proof}
Note that $|Z(G)| = p$ and  $\frac{G}{Z(G)} \cong {\mathbb{Z}}_p \times {\mathbb{Z}}_p$. Hence the  result follows from Theorem \ref{main2}.
\end{proof}

\begin{corollary}\label{4-cent}
If $G$ is a finite $4$-centralizer group then $G$ is commuting integral.     
\end{corollary}
\begin{proof}
If $G$ is a finite $4$-centralizer group then by Theorem 2 of \cite{bG94} we have  $\frac{G}{Z(G)} \cong {\mathbb{Z}}_2 \times {\mathbb{Z}}_2$. Therefore, by Theorem \ref{main2},
\[
\spec(\Gamma_G) = \{(-1)^{3(|Z(G)| - 1)}, (|Z(G)| - 1)^3\}.
\]
This shows that $G$ is commuting integral.
\end{proof}

\noindent Further, we have the following result.

\begin{corollary}
If $G$ is a finite $(p+2)$-centralizer $p$-group, for any prime $p$, then  \[
\spec(\Gamma_G) = \{(-1)^{(p^2 - 1)|Z(G)| - p - 1}, ((p - 1)|Z(G)| - 1)^{p + 1}\}.
\]
Hence, $G$ is commuting integral.
\end{corollary}
\begin{proof}
If $G$ is a finite $(p + 2)$-centralizer $p$-group then by Lemma 2.7 of \cite{ali00} we have  $\frac{G}{Z(G)} \cong {\mathbb{Z}}_p \times {\mathbb{Z}}_p$. Now the  result follows from Theorem \ref{main2}.
\end{proof}



 The following theorem shows that $G$ is commuting integral if the central factor of $G$ is isomorphic to the dihedral group $D_{2m} = \{a, b : a^{m} = b^2 = 1, bab^{-1} = a^{-1}\}$. 
\begin{theorem}\label{main4}
Let $G$ be a finite group such that $\frac{G}{Z(G)} \cong D_{2m}$, for $m \geq 2$. Then
\[
\spec(\Gamma_G) = \{(-1)^{(2m - 1)|Z(G)| - m - 1}, (|Z(G)| - 1)^m, ((m - 1)|Z(G)| - 1)^1\}.
\]
\end{theorem}
\begin{proof}
Since $\frac{G}{Z(G)} \cong D_{2m}$ we have $\frac{G}{Z(G)} = \langle xZ(G), yZ(G) : x^2, y^m,  xyx^{-1}y\in Z(G)\rangle$, where $x, y \in G$ with $xy \ne yx$.
It is not difficult to see that  for any $z \in Z(G)$, 
\[
 C_G(y) = C_G(y^iz) = Z(G) \sqcup yZ(G) \sqcup\cdots \sqcup  y^{m - 1}Z(G), 1 \leq i \leq m - 1
\]
  and
\[
C_G(xy^j) = C_G(xy^jz) = Z(G)  \sqcup xy^jZ(G), 1 \leq j \leq m
\]
are the only  centralizers of non-central elements of $G$. Also note that these centralizers are abelian subgroups of $G$. Therefore 
\[
\Gamma_G = K_{|C_G(y)\setminus Z(G)|} \sqcup (\underset{j = 1}{\overset{m}{\sqcup}} K_{|C_G(xy^j)\setminus Z(G)|}).
\]
 Thus $\Gamma_G = K_{(m - 1)n} \sqcup (\underset{j = 1}{\overset{m}{\sqcup}} K_n)$,  since  $|C_G(y)| = mn$ and $|C_G(x^jy)| = 2n$ for $1 \leq j \leq m$, where $|Z(G)| = n$. Hence the result follows.
\end{proof}

\begin{corollary}\label{5-cent}
If $G$ is a finite $5$-centralizer  group then $G$ is commuting integral. 
\end{corollary}
\begin{proof}
If $G$ is a finite $5$-centralizer group then by Theorem 4 of \cite{bG94} we have  $\frac{G}{Z(G)} \cong {\mathbb{Z}}_3 \times {\mathbb{Z}}_3$ or $D_6$. Now, if $\frac{G}{Z(G)} \cong {\mathbb{Z}}_3 \times {\mathbb{Z}}_3$ then  by Theorem \ref{main2} we have
\[
\spec(\Gamma_G) = \{(-1)^{8|Z(G)| - 4}, (2|Z(G)| - 1)^4\}.
\]
Again, if $\frac{G}{Z(G)} \cong D_6$ then  by Theorem \ref{main4} we have
\[
\spec(\Gamma_G) = \{(-1)^{5|Z(G)| - 4}, (|Z(G)| - 1)^3, (2|Z(G)| - 1)^1\}.
\]
In both the cases $\Gamma_G$ is integral. Hence $G$ is commuting integral.
\end{proof}

\noindent We also have the following result.
\begin{corollary}
Let $G$ be a finite non-abelian group and $\{x_1, x_2, \dots, x_r\}$ be a set of pairwise non-commuting elements of $G$ having maximal size. Then $G$ is commuting integral if $r = 3, 4$. 

\end{corollary}
\begin{proof}
By Lemma 2.4 of \cite{ajH07}, we have that $G$ is a $4$-centralizer or a $5$-centralizer group according as  $r = 3$ or $4$. Hence the result follows from Corollary \ref{4-cent} and Corollary \ref{5-cent}.  
\end{proof}

We now compute the spectrum of the commuting graphs of some well-known groups, using  Theorem \ref{main4}.   

\begin{proposition}\label{main05}
Let $M_{2mn} = \langle a, b : a^m = b^{2n} = 1, bab^{-1} = a^{-1} \rangle$ be a metacyclic group, where $m > 2$. Then 
\[
\spec(\Gamma_{M_{2mn}}) = \begin{cases}
&\!\!\!\!\!\!\!\!\{(-1)^{2mn - m - n - 1}, (n - 1)^m, (mn - n - 1)^1\} \text{ if $m$ is odd}\\
&\!\!\!\!\!\!\!\!\{(-1)^{2mn - 2n - \frac{m}{2} - 1}, (2n - 1)^{\frac{m}{2}}, (mn - 2n -1)^1\} \text{ if $m$ is even}. 
\end{cases} 
\]
\end{proposition}
\begin{proof}
Observe that $Z(M_{2mn}) = \langle b^2 \rangle$ or $\langle b^2 \rangle \cup a^{\frac{m}{2}}\langle b^2 \rangle$ according as $m$ is odd or even.  Also, it is easy to see that $\frac{M_{2mn}}{Z(M_{2mn})} \cong D_{2m}$ or $D_m$ according as $m$ is odd or even. Hence, the result follows from Theorem \ref{main4}.
\end{proof}

The above Proposition \ref{main05} also gives the spectrum of  the commuting graph of the dihedral group  $D_{2m}$, where $m > 2$, as given below:
\[
\spec(\Gamma_{D_{2m}}) = \begin{cases}
&\{(-1)^{m - 2}, 0^m, (m - 2)^1\} \text{ if $m$ is odd}\\
&\{(-1)^{\frac{3m}{2} - 3}, 1^{\frac{m}{2}}, (m - 3)^1\} \text{ if $m$ is even}. 
\end{cases} 
\]


\begin{proposition}
The spectrum of the commuting graph of the dicyclic group or the generalized quaternion group   $Q_{4m} = \langle a, b : a^{2m} = 1, b^2 = a^m,bab^{-1} = a^{-1}\rangle$, where $m \geq 2$, is given by
\[
\spec(\Gamma_{Q_{4m}}) = \{(-1)^{3m - 3},  1^m, (2m - 3)^1\}.
\]
\end{proposition}
\begin{proof}
The result follows from Theorem \ref{main4} noting that  $Z(Q_{4m}) = \{1, a^m\}$ and  $\frac{Q_{4m}}{Z(Q_{4m})} \cong D_{2m}$. 
\end{proof}

\begin{proposition}
Consider the group $U_{6n} = \{a, b : a^{2n} = b^3 = 1, a^{-1}ba = b^{-1}\}$. Then $\spec(\Gamma_{U_{6n}}) = \{(-1)^{5n - 4}, (n - 1)^3, (2n - 1)^1\}$.
\end{proposition}
\begin{proof}
Note that $Z(U_{6n}) = \langle a^2 \rangle$ and $\frac{U_{6n}}{Z(U_{6n})} \cong D_6$. Hence the result follows from Theorem \ref{main4}.
\end{proof}

We conclude the paper by noting that the   groups $M_{2mn}, D_{2m}, Q_{4m}$ and $U_{6n}$ are commuting integral.




\end{document}